\documentclass[]{article}
\ifdefined\pdfminorversion\pdfminorversion=7\relax\fi
\usepackage{mib}
\makeatletter
\let\ps@fancy\ps@plain
\providecommand{\@toptitlebar}{}
\providecommand{\@bottomtitlebar}{}
\makeatother
\usepackage{mathtools}
\usepackage[T2A]{fontenc}
\usepackage{color}
\usepackage[utf8]{inputenc} 
\usepackage[T1]{fontenc}    
\usepackage{hyperref}       
\usepackage{url}            
\usepackage{booktabs}       
\usepackage{amsfonts}       
\usepackage{nicefrac}       
\usepackage[nopatch=footnote]{microtype} 
\usepackage{lipsum}
\usepackage{graphicx}
\usepackage{color}
\usepackage{amsthm,amsmath,amssymb}
\usepackage{multirow}
\usepackage{listings}
\usepackage{mathrsfs}
\usepackage{abstract}
\usepackage{aliascnt}
\usepackage{cleveref}
\usepackage[backend=biber, style=numeric, url=true]{biblatex}
\usepackage{subfigure}
\usepackage{thm-restate}
\usepackage{setspace}

\newtheoremstyle{dot-after-num} 
  {}          
  {}          
  {\normalfont} 
  {}          
  {\bfseries} 
  {}          
  { }         
  {\thmname{#1}\thmnumber{ #2}.\thmnote{ (#3)}} 

\theoremstyle{dot-after-num}

\newcommand{\RegisterThm}[3]{%
  \newtheorem{#1}[#2]{#3}%
  \expandafter\providecommand\csname #1autorefname\endcsname{#3}%
  \crefname{#1}{#3}{#3s}%
  \Crefname{#1}{#3}{#3s}%
}

\RegisterThm{lemma}{theorem}{Lemma}
\RegisterThm{proposition}{theorem}{Proposition}
\RegisterThm{corollary}{theorem}{Corollary}
\RegisterThm{conjecture}{theorem}{Conjecture}
\RegisterThm{definition}{theorem}{Definition}
\RegisterThm{example}{theorem}{Example}
\RegisterThm{fact}{theorem}{Fact}
\RegisterThm{remark}{theorem}{Remark}

\numberwithin{equation}{section}

\newcommand{\moy}{{\Delta}_{(P,O)}}
\newcommand{\crow}{{\Delta}_{(P,o)}}

\title{On Correspondences between the Alexander Polynomials of Special Alternating Links and MOY Graphs}

\author{
 Leonardo Rodrigues de Medeiros \And Arker Oke Soe
}

\begin{document} 
\maketitle

\begin{abstract}

    Fox's conjecture famously asserts that the absolute values of the
    coefficients of the Alexander polynomial of alternating links are
    trapezoidal. In the setting of MOY graphs, where a different notion of
    Alexander polynomial appears, the equivalent result to Fox's conjecture is
    known to hold. In this paper, we relate the two polynomials in the case of
    special alternating links. More precisely, we show that their degrees,
    first, and last coefficient agree, and that the MOY polynomial coefficients
    always dominate the classical Alexander polynomial.
    
\end{abstract}

\section{Introduction}

A link is a smooth embedding of a finite disjoint union of circles in $S^3$. A
knot is a link with one component. A link projection is a planar
representation of a link with over-crossings and under-crossings identified. A
link projection is called \textit{alternating} if, while traversing along each
component, the crossings are encountered alternately as over-crossings and
under-crossings. Alternating links form one of the central families of knot
theory,  both because their diagrams are combinatorially rigid and because many
of their classical invariants admit especially concrete descriptions. 

One of the oldest such invariants is the Alexander polynomial
\cite{Alexander1928-oc}. For alternating links, Crowell gave a combinatorial
model for the Alexander polynomial with arborescences \cite{crowell_genus_1959}.
The coefficient sequence of the Alexander polynomial of an alternating link is
the subject of Fox's trapezoidal conjecture, which asserts that the
coefficients, disregarding the sign, are trapezoidal \cite{Fox1962}: they form a sequence $a_1,
\dotsc, a_n$ such that $a_i = a_{n+1-i}$, and there exists an integer $k$ with
$1 \le k \le \lceil n/2 \rceil$ such that 
\[
a_1 < \dotsb < a_k = \dotsb = a_{n+1-k} > a_{n-k} > \dotsb > a_n.
\]
 Although this conjecture remains open for the general
alternating case, it has been settled for many special cases.  The case of
two-bridge knots was settled by Hartley \cite{Hartley1979-vj}; the case for
alternating algebraic links by Murasugi \cite{Murasugi1985-ah}; for genus two
alternating knots (among others) by Ozsv\'ath and Szab\'o \cite{Ozsvath2003-wy}
(another proof for genus two alternating knots was given by Jong
\cite{Jong2009-fj}). Hafner, M\'esz\'aros, and Vidinas proved a stronger
log-concavity statement for special alternating links
\cite{hafner_log-concavity_2024}, and another proof of log-concavity for the
same case was given by K\'alm\'an, M\'esz\'aros, and Postnikov
\cite{Kalman2025-hc}. Moreover, for certain Murasugi sums of special alternating
links Azarpendar, Juhász, and Kálmán
\cite{azarpendar2024foxstrapezoidalconjecture} showed trapezoidality, while the
same result was proven via dimers by Mészáros, Sherman-Bennett, Vidinas
\cite{mészáros2025dimerviewfoxstrapezoidal}

A second Alexander-type polynomial arises from MOY graphs, which were first
introduced by Murakami, Ohtsuki, and Yamada to provide graphical ways of
computing $SU(n)$ quantum invariants for links \cite{MOY1998}.  Later, Bao and
Wu introduced an Alexander polynomial for MOY graphs \cite{Bao2020-ph}, and more
recently, Liao and Wu developed a spanning-tree model for the Alexander
polynomial of MOY graphs and proved the corresponding trapezoidality statement
for its coefficient sequence \cite{liao_clock_2024}.   

The goal of this paper is to compare the Alexander polynomial of a special
alternating link with the Alexander polynomial arising from the associated MOY
graph. For an alternating link projection $P$ (and a fixed choice of root $r$),
we denote by $\crow(t)$ the `positive' Alexander polynomial of $P$ coming from
Crowell's model (that is, $\crow(-t)$ is the Alexander polynomial), and by
$\moy(t)$ the associated MOY graph Alexander polynomial of $P$. 

Let
\begin{equation} \label{eq:crowell_poly}
  \crow(t) = \sum_{i=m_2}^{d_2} b_i t^i,
\end{equation}
and
\begin{equation} \label{eq:moy_poly}
  \moy(t) = \sum_{i=m_1}^{d_1} a_i t^i.
\end{equation}
where $a_{m_1}, a_{d_1}, b_{m_2}, b_{d_2}$ are all non-zero. The following theorem
is the main result of this paper.

\begin{restatable}{theorem}{maintheorem}
\label{thm:main_thm}
If $P$ is an alternating projection of a special alternating link and $\crow(t), \moy(t)$ are as above, then
    \begin{enumerate}
        \item $m_1 = m_2$ and $d_1 = d_2$,
        \item $b_i \le a_i$ for every $m_1 \le i \le d_1$,
        \item $a_{d_1} = a_{m_1} = b_{m_2} = b_{d_2}$.
    \end{enumerate}
\end{restatable}

 The proof is purely combinatorial. Both polynomials are expressed as sums over rooted arborescences of the planar graph underlying $P$, but with different orientations and, a priori, different weightings. The special alternating hypothesis has two key consequences. First, the Crowell and MOY weightings agree on every edge of $P$. Second, each Seifert circle lies entirely in one of the two distinguished subgraphs determined by the edge weights. These facts allow us to construct a weight-preserving injective map
\[
m:T(o,r)\longrightarrow T(O,r)
\]
from arborescences in the alternating orientation to arborescences in the link orientation. This injection gives the coefficientwise domination in \Cref{thm:main_thm}. A more refined analysis of extremal-weight arborescences shows that the map is bijective on the arborescences contributing to the first and last nonzero coefficients, giving the equality of the extremal coefficients.

The paper is structured as follows. In \Cref{sec:prelim} we recall important
definitions regarding orientation, the MOY and Crowell weightings, and Seifert
circles in the graph underlying $P$. In \Cref{sec:weight} we show that the MOY
and Crowell weightings agree on the special alternating case. In \Cref{sec:map},
we construct maps $m$ and $m^*$ between arborescences of $P$ with respect to the
link orientation and arborescences of $P$ with respect to the alternating
orientation. In \Cref{sec:polyn} we prove properties of the maps $m$ and $m^*$
which we use to give a proof of \Cref{thm:main_thm}.

\section{Preliminaries}
\label{sec:prelim}

\subsection{Orientations and weightings on an alternating link projection}

Throughout this section, we follow the notation of Crowell's landmark paper
\cite{crowell_genus_1959}.

Consider a non-trivial alternating link $L$ with an alternating projection $ P
$. $ P $ can be seen as a planar graph with crossings corresponding to vertices
and segments between crossings corresponding to edges. As in Crowell
\cite{crowell_genus_1959}, this graph inherits two natural orientations (see
\Cref{fig:orientations_on_trefoil}):
\begin{enumerate}

  \item A link orientation $ O $ coming from the orientation of the original link.

  \item An alternating orientation $ o $, where edges are oriented from
  overcrossings to undercrossings.

\end{enumerate}

\begin{figure}[!ht]
    \centering
    \includegraphics[width=0.5\linewidth]{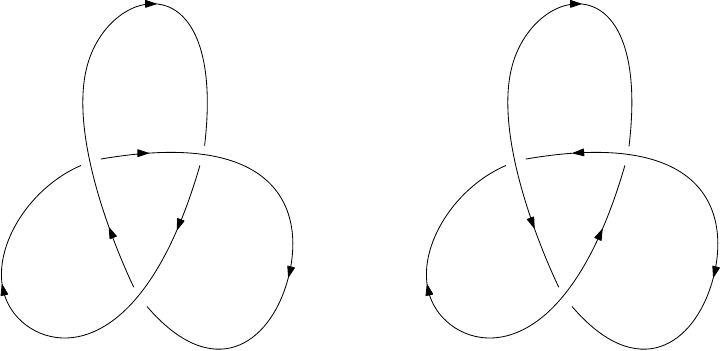}
    \caption{Link orientation (left) and alternating orientation (right) of the
    trefoil}
    \label{fig:orientations_on_trefoil}
\end{figure}

Let $ (P,O) $ and $ (P,o) $ be the two corresponding oriented planar graphs.
Denote the set of edges of $P$ by $ E(P) $. 

\begin{definition}
    A \textbf{weighting} $\alpha$ of a graph $G$ is the data of a function
    \begin{align*}
       \alpha: E(G) \to  \{1,t\}.
    \end{align*}
\end{definition}

In our graph $P$, we can consider two weightings (all mentions of orientations
will refer to $O$):

\begin{enumerate}
    \item (MOY weighting) $\mu$, where at each vertex the first and second
    incoming edges, in counterclockwise order, get weights $1$ and $t$,
    respectively.
    
    \item (Crowell weighting) $\kappa$, where at each vertex we traverse the over-segment and assign weight $1$ to the under crossing on the left and $t$ to the under crossing on the right.
    
\end{enumerate}
See \Cref{fig:weightings} for a comparison of the weightings $\mu$ and $\kappa$
near a crossing.

\begin{figure}[!ht]
    \centering
    \includegraphics[width=0.5\linewidth]{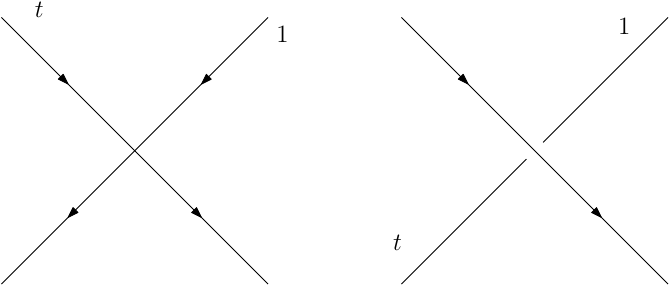}
    \caption{Weightings $\mu$ (left) and $\kappa$ (right) near a crossing}
    \label{fig:weightings}
\end{figure}

\begin{definition}
    Let $(G, \partial)$ be an oriented graph and $r$ a vertex of $G$. A subgraph
    $A$ of $G$ is an \textbf{arborescence on $(G, \partial)$ rooted at a vertex
    $r$} if it contains all vertices of $G$ and for every vertex $v$ there is a
    unique directed path from $r$ to $v$ in $ A $ agreeing with the orientation
    $\partial$.
\end{definition}
When the graph $G$ is clear from context, we write ${\mathcal{T}(\partial, r)}$
for the set of arborescences on $(G, \partial)$ rooted at $r$.

For a subgraph $S$ and a weighting $\alpha$ of $G$, define the
\textbf{weighting of $S$} as
\[
\alpha(S) := \prod_{i \in E(S)} \alpha(i).
\]

Fixing a vertex $r$ of $P$, the orientations $O$ and $o$ define polynomials
\[
\Delta_{(P,O)}(t) := \sum_{A \in \mathcal{T}(O,r)} \mu(A)
\quad \text{and} \quad
\Delta_{{(P,o)}}(t) := \sum_{A \in \mathcal{T}(o,r)} \kappa(A).
\]

Theorem 2.12 of \cite{crowell_genus_1959} establishes that $\Delta_{(P,O)}(-t)$
is (up to a factor of $t^{\pm k}$) the Alexander polynomial of $P$. On the other
hand, $\Delta_{(P,O)}(t)$ is the Alexander Polynomial associated to the MOY
graph $(P,O)$ with all edges colored 1 as in \cite{liao_clock_2024}.

The idea for the proof of \Cref{thm:main_thm} will be to construct a
weight-preserving injective map
\[
m : \mathcal{T}(o, r) \to \mathcal{T}(O, r)
\]
such that $\kappa(A) = \mu(m(A))$ and every arborescence of minimal/maximal
weight in $\mathcal{T}(O,v)$ is in the image of $m$.

\subsection{Seifert circles and special Seifert circles}

Let $\mathcal{C}(G, \partial)$ be the set of oriented cycles of the oriented
graph $(G,\partial)$.

\begin{definition}[Crowell (4.7) \cite{crowell_genus_1959}]
   Let $P$ be an alternating projection of an alternating link. The
   \textbf{Seifert circles of $P$} are oriented cycles lying on the intersection
   \[
   \mathscr{F}(P) := \mathcal{C}(P,O) \cap \mathcal{C}(P,o).
   \]
\end{definition}

As discussed by Crowell \cite{crowell_genus_1959}, each edge of $P$ lies in
exactly one Seifert circle.

\begin{definition}
A \textbf{special Seifert circle} is a Seifert circle where either the inside or
the outside region delimited by the cycle contains no vertices or edges of $P$.
The set of special Seifert circles of $P$ is denoted $\mathscr{F}_s(P)$.
\end{definition}

Let $H, K$ be the two subgraphs of $P$ both containing all of its vertices, and
edges $\kappa^{-1}(1)$ and $\kappa^{-1}(t)$, respectively. An alternative
characterization of special Seifert circles is given in (4.7)
\cite{crowell_genus_1959} as
\[
\mathscr{F}_s(P) = \mathcal{C}(H,o) \cup \mathcal{C}(K,o).
\]

\begin{definition}
    An alternating link projection $P$ is defined to be \textbf{special
    alternating} if
    \[
    \mathscr{F}(P) = \mathscr{F}_s(P).
    \]
\end{definition}
For a special alternating link projection, $H$ and $K$ will contain, with
respect to the link orientation $O$, all the counterclockwise and clockwise
circles, respectively.

\section{The weightings \texorpdfstring{$\mu$}{μ} and \texorpdfstring{$\kappa$}{κ} in the special alternating case} 
\label{sec:weight}
Before constructing our maps $m$ and $m^*$, we show that for special alternating
links the weightings $\mu$ and $\kappa$ agree. 

Recall that a \textbf{positive alternating link projection} is an alternating
link projection where all crossings are positive, meaning that when traversing
the over-segment via the orientation $O$, the under-segment is oriented 
from right to left.

\begin{lemma}\label{lem:weightings_agree}
    If $P$ is a positive alternating link projection, the weightings $\mu$ and
    $\kappa$ agree on every edge of $P$.
\end{lemma}
\begin{proof}
Over the course this proof we always refer to the link orientation $O$. 

As all crossings of $P$ are positive, we proceed by analyzing the weights of the
edges of a single positive crossing. Knowing that every edge is incoming to some
crossing, it is sufficient to analyze the incoming edges.

For the weighting $\mu$, we know that the first incoming edge has weight
$1$ and the second incoming edge has weight $t$. By the positivity of the link,
this means that the incoming over-edge has weight $1$ and the incoming
under-edge has weight $t$.

\begin{figure}[!ht]
    \centering
    \includegraphics[width=0.4\linewidth]{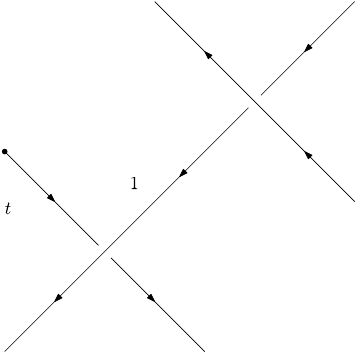}
    \caption{Weightings $\mu$ and $\kappa$ match}
    \label{fig:weightings_match}
\end{figure}

For the weighting $\kappa$, we can immediately see that the incoming under-edge
will have weight $t$ by the positivity of the link (it will be to the right of
the over-segment). To decide on the weight of the incoming over-edge, we need to
look at the previous crossing this edge has intersected. Using positivity again,
we conclude that the incoming over-edge has weight $1$ (as it will be to the
left of the over-segment in this previous crossing).

With the orientations agreeing on the incoming edges of every crossing, we are
done.
\end{proof}

The same case analysis shows that the weightings agree for any negative
alternating link projection. As every special alternating link projection is
either positive or negative (see \cite{murasugiJonesPolynomialsClassical1987},
Corollary 7) we get:
\begin{corollary}
    If $P$ is a special alternating link projection, the weightings $\mu$ and
    $\kappa$ agree on every edge of $P$.
\end{corollary}

\section{The maps \texorpdfstring{$m^*$}{m*} and \texorpdfstring{$m$}{m}}
\label{sec:map}

We now construct the maps $m : \mathcal{T}(o, r) \to \mathcal{T}(O, r)$ and $m^*
: \mathcal{T} (O,r) \to \mathcal{T}(o,r)$ with the properties that $\mu(A) =
\kappa ( m^*(A)), \kappa(A) = \mu(m(A)),$ and $m^* \circ m = \text{id}_{\mathcal{T}(o,r)}$.

\subsection{Construction of \texorpdfstring{$m^*$}{m*}}

The map $m^*$ is defined using the following two lemmas. Recall that
an \textit{alternating dimap} is a planar Eulerian digraph such that the edges
around each vertex are oriented in an alternating clockwise incoming-outgoing
order. The graph $(P, o)$ is an alternating dimap whose clockwise-oriented
regions are precisely the Seifert circles of $P$.

For an alternating dimap $D$, denote the cycles surrounding the
clockwise-oriented regions by $C_1, \ldots, C_k$. For a subset $S$ of the edges
of $D$, let $c_i(S)$ be the number of edges of $S$ in the cycle $C_i$.

\begin{lemma}[{\cite[Proof of Theorem 10.1]{KALMAN2013823}, \cite[Lemma 4.3]{hafner_log-concavity_2024}}]\label{lem:meszaros_lemma}
     Let $D$ be an alternating dimap, $T$ be any spanning tree in $D$, and fix
     any $r \in V(D)$. Then, there exists an arborescence $A$, rooted at $r$,
     such that $c_i(A)=c_i(T)$ for all $i \in [k]$.
\end{lemma}

\begin{lemma}[{\cite[Proof of Theorem 10.1]{KALMAN2013823}, \cite[Lemma
4.6]{hafner_log-concavity_2024}}]\label{lem:uniqueness}
    Let $D$ be an alternating dimap, $T$ be any spanning tree
    of $D$, and fix a vertex $r$ of $D$. For any sequence $\{ x_1, \dots, x_n
    \},$ there exists at most one arborescence $A$ rooted at $r$ such that
    $c_i(A) = x_i$ for all $i.$
\end{lemma}

Together, \Cref{lem:meszaros_lemma} and \Cref{lem:uniqueness} show that 
for each arborescence $A' \in \mathcal{T}(O, r)$, there is a \emph{unique}
arborescence in $\mathcal{T}(o, r)$ whose edge counts in each Seifert circle
match those of $A'$. We define
\[
    m^*(A') := \text{the unique } B \in \mathcal{T}(o, r) \text{ such that } 
    c_i(B) = c_i(A') \text{ for all } i \in [k],
\]
where the existence of $B$ follows from applying Lemma~\ref{lem:meszaros_lemma} to $(P,o)$ 
with input tree $A'$ (viewed as an unoriented spanning tree), and uniqueness follows from \Cref{lem:uniqueness}.

\subsection{Construction of \texorpdfstring{$m$}{m}}

For the case of $(P,O)$, the proof and construction for \Cref{lem:uniqueness}
work identically. We now let $C_1, \dotsc, C_k$ be the Seifert circles of $P$, and
for a subset $S$ of edges of $P$ we let $c_i(T)$ be the number of edges of $S$
in $C_i$.

\begin{lemma}\label{lemma:seifert_version}
    Let $P$ be the projection of a special alternating link. Let $T$ be any
    (non-oriented) spanning tree in $P$, and fix any $r \in V(D)$. Then, there
    exists an arborescence $A \in \mathcal{T}(O,r)$, such that $c_i(A)=c_i(T)$
    for all $i \in [k]$.
\end{lemma}

\begin{proof}
We adapt the construction of \cite{hafner_log-concavity_2024} to $(P, O)$. Given a spanning tree $T$ rooted 
at $r$, we iteratively modify it into an arborescence on $(P, O)$ while preserving the 
number of edges in each Seifert circle.

Define the set of \textit{good vertices} $V'(T) \subseteq V(P)$ as follows. We
place $r \in V'(T)$ if and only if $r$ is not the terminal vertex of any edge of
$T$. For any other vertex $v$, we place $v \in V'(T)$ if and only if every
vertex along the unique (undirected) path from $r$ to $v$ in $T$ is already in
$V'(T)$, and $v$ is the terminal vertex of exactly one edge of $T$, namely the
edge on this path incident to $v$.

If $T$ is not yet an arborescence on $(P, O)$, there exists a vertex $v_1 \notin
V'(T)$ such that the unique path from $r$ to $v_1$ in $T$ passes only through
good vertices. Such a vertex is necessarily the terminal vertex of some edge
$e_1$ not lying on the path from $r$ to $v_1$. Let $V_1$ denote the set of
vertices whose unique path to $r$ in $T$ passes through $e_1$. Since $T$ is a
spanning tree, there exists an edge $e_2$ of $(P, O)$ with terminal vertex in
$V_1$ and initial vertex not in $V_1$. Replacing $e_1$ with $e_2$ yields a new
spanning tree $T_1$.

This replacement preserves the edge count in every Seifert circle: since each
Seifert circle of $P$ lies entirely within $H$ or within $K$ by the special
alternating hypothesis, and since both $e_1$ and $e_2$ have their terminal
vertex in $V_1$, they lie in the same Seifert circles, giving $c_i(T_1) =
c_i(T)$ for all $i \in [k]$.

As the number of good vertices strictly increases at each step, the process
terminates at an arborescence $A \in \mathcal{T}(O, r)$ with $c_i(A) = c_i(T)$
for all $i \in [k]$. 
\end{proof}

Unlike the alternating dimap setting, \Cref{lem:uniqueness} does not apply to
$(P, O)$, since the Seifert circles of $P$ are not all clockwise-oriented under
$O$. Consequently, the construction of \Cref{lemma:seifert_version} may yield
different arborescences depending on the order in which vertices $v_1$ (as in
the proof of \Cref{lemma:seifert_version}) are chosen. To obtain a well-defined
map $m : \mathcal{T}(o,r) \to \mathcal{T}(O,r)$, we fix once and for all an
enumeration of the vertices of $P$, and at each step of the construction, we
choose the eligible vertex $v_1$ of lowest index. With this convention, define
\begin{align*} m(A) := \ & \text{the arborescence in  $\mathcal{T}(O, r)$
produced by applying the construction of} \\ &
\text{\Cref{lemma:seifert_version} to  $A \in \mathcal{T}(o,r)$, with the
lowest-index tie-breaking rule throughout.} \end{align*} This gives a
well-defined map $m \colon \mathcal{T}(o, r) \to \mathcal{T}(O, r)$ satisfying
$c_i(m(A)) = c_i(A)$ for all $i \in [k]$.

\section{The polynomials \texorpdfstring{$\crow (t)$}{Δ\_(P,o)(t)} and \texorpdfstring{$\moy (t)$}{Δ\_(P,O)(t)}}
\label{sec:polyn}

We now prove the properties of the maps $m, m^*$ we will need for the proof of
\Cref{thm:main_thm}.

\begin{lemma}\label{lemma:map_preserves_weights}
    For any $A \in \mathcal{T}(o,r)$, $A' \in \mathcal{T}(O,r)$, 
    $\kappa(A) = \mu(m(A))$ and $\mu(A') = \kappa(m^*(A'))$.
\end{lemma}

\begin{proof}
By \Cref{lemma:seifert_version} the number of edges of $A$ and $m(A)$ lying on
$K$ are the same, as the number of edges around each Seifert circle does not
change and Seifert circles lie entirely in $H$ or in $K$ by the special
alternating hypothesis. As the weighting of a tree with respect to $\kappa$ is
determined by the number of its edges lying in $K = \kappa^{-1}(t)$ we have
\[
\kappa(A) = \kappa(m(A)).
\]
By \Cref{lem:weightings_agree}
\[
 \kappa(m(A)) = \mu(m(A)).
\]
The proof for $A'$ and $m^*$ is identical.
\end{proof}

Recall from \eqref{eq:moy_poly} and \eqref{eq:crowell_poly} that 
\[
    \Delta_{(P,O)}(t) = \sum_{i=m_1}^{d_1} a_i t^i
    \qquad \text{and} \qquad
    \Delta_{(P,o)}(t) = \sum_{i=m_2}^{d_2} b_i t^i.
\]
\Cref{lemma:map_preserves_weights} implies part 1 of Theorem \ref{thm:main_thm}:
\begin{corollary}\label{cor:equal_degree}
    $m_1 = m_2$ and $d_1 = d_2$.
\end{corollary}

\begin{proof}
    Lemma \ref{lemma:map_preserves_weights} guarantees that from any
    arborescence $A$ on $\mathcal{T}(O,r)$ of weight $t^k$ we can find an
    arborescence $m(A)$ on $\mathcal{T}(o,r)$ of same weight, and vice versa. As
    $t^{m_1}$ and $t^{m_2}$ are the smallest weights for arborescence in
    $\mathcal{T}(O,r)$ and $\mathcal{T}(o,r)$, respectively, they have to agree
    and thus $m_1 = m_2$. A symmetric argument for highest weight arborescences
    in $\mathcal{T}(O,r)$ and $\mathcal{T}(o,r)$ gives $d_1 = d_2$.
    
\end{proof}

\begin{lemma}\label{lem:m_inj}
    The map $m$ is injective.
\end{lemma}

\begin{proof}
    Let $C_1,C_2, \dots , C_n$ be the Seifert circles of $P$, $A$ and $A'$ be
    distinct arborescences in $\mathcal{T}(P,o)$, and $a_i(A), a_i(A')$ be
    the number of edges of $A$ and $A'$ in $C_i$, respectively. Then by
    \Cref{lem:uniqueness} there exists at least one $i \in [n]$ such
    that $c_i(A) \neq c_i(A').$ Finally, since $m$ preserves the number of edges
    in each Seifert circle, $m(A)$ and $m(A')$ must have different number of
    edges in $C_i$, thus $m(A) \ne m(A')$.
\end{proof}

We now show part 2 of Theorem \ref{thm:main_thm}:
\begin{corollary}\label{cor:dominates}
    The coefficients of  $\moy (t)$ dominate the coefficients of $\crow (t)$
\end{corollary} \label{dominate}

\begin{proof}
    For any $m_1 \le i \le d_1$, $a_i$ is by definition the number of distinct
    arborescences in $\mathcal{T}(O,r)$ with weight $t^i$. Similarly, for $m_2
    \le j \le d_2$, $b_j$ is the number of distinct arborescences in
    $\mathcal{T}(o,r)$ with weight $t^j$. Notice that (1) $m_1=m_2$, $d_1 = d_2$
    by \Cref{cor:equal_degree}, (2) $m : \mathcal{T}(o,r) \to \mathcal{T}(O,r) $
    is injective by \Cref{lem:m_inj}, and (3) $m$ preserves the weights $\mu =
    \kappa$ by \Cref{lemma:map_preserves_weights}. Conditions (1), (2), and (3)
    together imply that the number of arborescences on $\mathcal{T}(o,r)$ of
    weight $t^i$ cannot exceed the number of arborscences on $\mathcal{T}(O,r)$
    of weight $t^i$ for any $m_1 = m_2 \le i \le d_1 = d_2$. This shows $b_i \le
    a_i$ for every $i$.
\end{proof}
To prove part 3 of Theorem \ref{thm:main_thm}, we will have to show that $m$
induces a bijection between the set of arborescences of minimal weight in
$\mathcal{T}(o,r)$ and $\mathcal{T}(O,r)$. The next result characterizes those
arborescences:

\begin{lemma}\label{lemma:charac_minimal}
    The arborescences $A \in \mathcal{T}(o,r)$ of minimum weight are the ones missing exactly one edge per Seifert circle in $\mathcal{C}(H,o)$.
\end{lemma}

\begin{proof}
    Let $A \in \mathcal{T}(o,r)$ be arbitrary. As every tree on a graph contains
    the same number of edges, the weight $\kappa(A)$ is completely determined by
    the number of edges of $A$ in $H$. Being a tree, $A$ cannot contain a cycle,
    so $A$ is missing at least one edge per Seifert circle in
    $\mathcal{C}(H,o)$. Furthermore, if $A$ misses exactly one edge per Seifert
    circle in $H$, it is necessarily an arborescence of minimum weight. 
    
    We now show we can construct an arborescence missing one edge per Seifert
    circle in $\mathcal{C}(H,o)$. Start with the set of edges $E(H)$, and remove
    one edge of each circle in $\mathcal{C}(H,o)$, so the result is a forest on
    $P$. Extend this forest to a spanning tree $A$ (only adding edges on $K$),
    and use Lemma~\ref{lemma:seifert_version} to construct an arborescence $A'
    \in \mathcal{T}(o,r)$ which still has the same number of edges per Seifert
    circle as $A$. $A'$ is thus an arborescence missing one edge per Seifert
    circle in $\mathcal{C}(H,o)$, as desired.
\end{proof}

We denote the set of arborescences of $(P,O)$ of minimal weight by
$\mathcal{T}_m(O,r)$ and the set of arborescences of $(P,o)$ of minimal weight
by $\mathcal{T}_m(o,r)$. 

\begin{lemma}
    The map $m^*$ induces a bijection between the arborescences of minimal
    weight of $\mathcal{T}_m(O,r)$ and the arborescences of minimal weight of
    $\mathcal{T}_m(o,r)$.
\end{lemma}

\begin{proof}

    By \Cref{lemma:charac_minimal}, the arborescences in $\mathcal{T}_m(O,r)$
    are the ones that are missing exactly one edge of every cycle in $H$. For
    an arbitrary arborescence $A \in \mathcal{T}_m(O,r)$, we consider for each
    circle $C_i \in H$ the unique vertex $v_i$ with one outgoing edge but no
    incoming edges from $C_i.$ Then, for each $v_i,$ we remove the outgoing edge
    of $v_i$ from $A$ and replace it with the incoming edge to get $A'$ (the
    procedure is exemplified in \Cref{fig:placeholder}). We show that the edge set
    of $A'$ is an arborsecence in $\mathcal{T}_m(o,r)$. 
    
    Recall that as $A$ is an arborescence in $(P,O)$, for every vertex $v \in P$
    there is a unique path from $r$ to $v$ in $A$, and the direction of the path
    agrees with the orientation $O$. Let $e_1, e_2, \dotsc, e_n$ be the sequence
    of edges in the path from $r$ to $v$ (edges colored in red on left drawing in \Cref{fig:placeholder}).

    We claim that if $e_k$ is the first edge in the path contained in the
    Seifert circle $C_i \in H$, then it must be that $e_k$ is the outgoing edge
    of $v_i$. Indeed, if not, we could use the edges $e_1, \dotsc, e_{k-1}$ and
    the incoming edge of the starting vertex of $e_k$ to build a non-directed
    path from $r$ to a vertex, contradicting that $A$ is an arborescence. Also
    notice that the path can only intersect $H$ on a single consecutive sequence
    of edges of the path, as otherwise we would find a cycle in $A$.

    Notice that $A'$ is still a spanning tree of $P$, as the number of
    edges and the connectivity within each Seifert circle remains the same as in
    $A$. To show that the unique path from $r$ to a vertex $v$ agrees with the
    orientation, we start with the path $e_1, e_2, \dotsc, e_n$ from $r$ to $v$
    in $A$ and modify it to get a path in $A'$, showing that the edges of the
    new path agree with the orientation $o$.

    The edges of $A$ in $K$ have the same orientation in $O$ and $o$, and were
    not modified when constructing $A'$, so we will not modify them in the path.
    Let $C_i \in H$ be a Seifert circle in $H$ intersecting the path. Substitute
    the edges of $C_i$ on the path by all the edges of $C_i$ \emph{not}
    intersecting the original path (see the modified path on the right in \Cref{fig:placeholder}). Our previous claim that the first edge of the path
    intersecting $C_i$ is the outgoing edge of $v_i$ guarantees that the
    substituting edges are all in $A'$. Performing the above path modification
    for each Seifert circle intersecting the original path results in a directed
    path (with respect to $o$) from $r$ to $v$ using edges of $A'$, showing that
    $A'$ is an arborescence in $\mathcal{T}_m(o,r)$ as we wanted.

    \begin{figure}[!ht]
    \centering
    \includegraphics[width=0.8\linewidth]{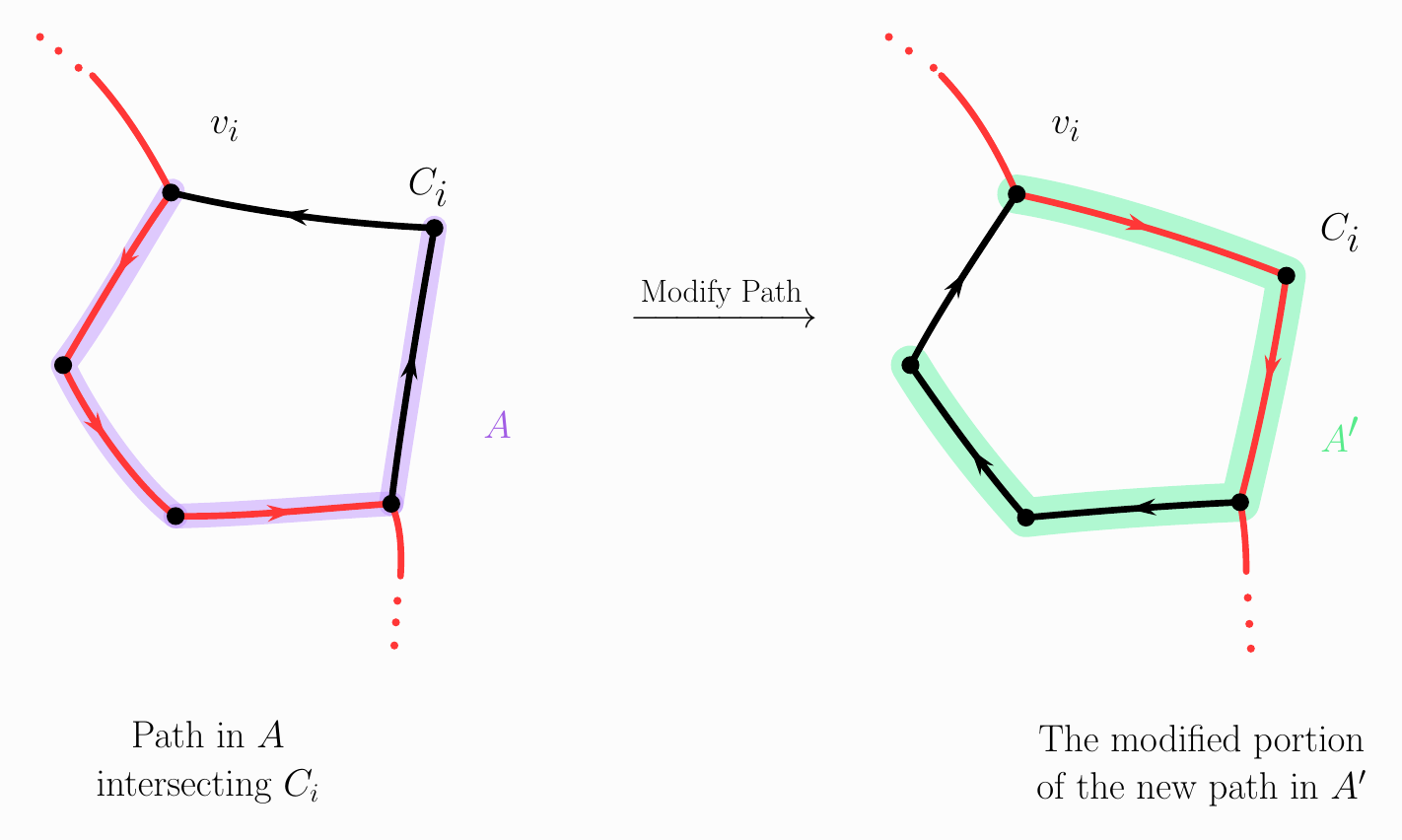}
    \caption{Path modification procedure on the Seifert circle $C_i \in H$. On
    the left, we color in red the original directed path from $r$ to $v$ (using
    the orientation $O$) and shade in purple the edges of $C_i$ lying in $A$. On
    the right, we still color in red the modified portion of the directed path
    from $r$ to $v$ (now using the orientation $o$), and shade in green the
    edges of $C_i$ lying in $A'$.}
    \label{fig:placeholder}
\end{figure}
    
    This procedure takes two distinct arborescence in $\mathcal{T}_m(O,r)$ to
    two distinct arborescences in $\mathcal{T}_m(o,r)$, so we conclude that
    $|\mathcal{T}_m(O,r)| \leq |\mathcal{T}_m(o,r)|$. Since we have shown the
    reverse inequality in \ref{dominate}, we know that $|\mathcal{T}_m(O,r)| =
    |\mathcal{T}_m(o,r)|.$ Since $m$ is injective, this shows that it is also a
    bijection. 
\end{proof}

\begin{corollary}\label{corollary:minmaxequal}
    The coefficients of the minimum and maximum degree terms of $\moy (t)$ and $\crow (t)$ are equal.  
\end{corollary}

\begin{proof}
    As is Corollary~\ref{cor:dominates}, we know that the coefficients of the
    minimum degrees come from the number of trees with minimum weight, so by
    above lemma the coefficients are equal. By the palindromicity of the
    coefficients of the two polynomials, the same applies to the coefficient of
    the maximum degree.

\end{proof}

We now prove the main result of our paper.

\maintheorem*
\begin{proof}
    The first part follows by Corollary \ref{cor:equal_degree}, the second
    part by Corollary \ref{cor:dominates}, and the third by Corollary
    \ref{corollary:minmaxequal}
\end{proof}

\section{Conclusion}

We have shown that for any alternating projection $P$ of a special alternating
link, the polynomials $\Delta_{(P,O)}(t)$ and $\Delta_{(P,o)}(t)$ satisfy three
properties: they share the same degree range; the coefficients of
$\Delta_{(P,O)}(t)$ dominate those of $\Delta_{(P,o)}(t)$; and their extremal
coefficients agree. The proof proceeds by constructing a weight-preserving
injective map $m \colon \mathcal{T}(o,r) \to \mathcal{T}(O,r)$, with the special
alternating hypothesis ensuring both that the weightings $\mu$ and $\kappa$
agree and that the arborescences of minimal and maximal weight are in
bijection.

The restriction to special alternating links is used in two essential ways:
first, to guarantee that $\mu$ and $\kappa$ agree on all edges of $P$
(\Cref{lem:weightings_agree}), and second, to ensure that each Seifert circle
lies entirely in $H$ or in $K$ (which underpins the weight computation in
\Cref{lemma:map_preserves_weights}). Removing either of these hypotheses appears
to require substantially new ideas.

Motivated computational evidence (around 300 examples from the knot database in
\cite{pyknotid}) on alternating links beyond the special alternating case, we
conjecture that the conclusion of \Cref{thm:main_thm} holds for all alternating
links.

\begin{conjecture}
Let $L$ be an alternating link with alternating projection $P$, and let
\[
    \Delta_{(P,O)}(t) = \sum_{i=m_1}^{d_1} a_i t^i
    \qquad \text{and} \qquad
    \Delta_{(P,o)}(t) = \sum_{i=m_2}^{d_2} b_i t^i
\]
where $\crow(t)$ is the Alexander polynomial of a link coming from Crowell's model, and by $\moy(t)$ is the associated MOY graph Alexander polynomial of the link. Then, 
\begin{enumerate}
    \item $m_1 = m_2$ and $d_1 = d_2$,
    \item $b_i \leq a_i$ for every $m_1 \leq i \leq d_1$,
    \item $a_{d_1} = a_{m_1} = b_{m_2} = b_{d_2}$.
\end{enumerate}
\end{conjecture}

A proof of this conjecture in full generality would likely require either extending the map 
$m$ to the non-special setting, where the weightings $\mu$ and $\kappa$ need no longer 
agree, or developing an alternative approach. We leave this as an open problem.

\section{Acknowledgements}

We would like to warmly thank our research advisor, Professor Karola
M\'esz\'aros. This work would not have been possible without her patient and
thoughtful guidance at every step of our research journey, from ideation to
proofreading the paper.  

\printbibliography

\end{document}